\documentclass[10pt,a4paper]{amsart}

\usepackage{graphicx,amssymb,color,amscd}
\usepackage{tikz}

\theoremstyle{plain}
\newtheorem{theorem}{Theorem}[section]
\newtheorem{lemma}[theorem]{Lemma}

\theoremstyle{definition}

\newtheorem{thm}{Theorem}
\theoremstyle{remark}

\newtheorem{claim}[theorem]{Claim}
\newtheorem*{acknowledgments}{Acknowledgments}
\numberwithin{equation}{section}
\numberwithin{figure}{section}

\newcommand{\bd}{\begin{description}}   
\newcommand{\ed}{\end{description}} 
\newcommand{\ba}{\begin{array}}      \newcommand{\ea}{\end{array}} 
\newcommand{\bc}{\begin{center}}     \newcommand{\ec}{\end{center}} 
\newcommand{\be}{\begin{enumerate}}  \newcommand{\ee}{\end{enumerate}} 
\newcommand{\beq}{\begin{eqnarray}}  \newcommand{\eeq}{\end{eqnarray}} 
\newcommand{\beQ}{\begin{eqnarray*}} \newcommand{\eeQ}{\end{eqnarray*}} 
\newcommand{\bi}{\begin{itemize}}    \newcommand{\ei}{\end{itemize}}

\newcommand{\dessin}[2]{
  \vcenter{\hbox{\includegraphics[height=#1]{#2.pdf}}}}

\begin{document} 
\title[Linking, framing and the Kontsevich integral]{Linking coefficients and the Kontsevich integral} 
\author[J.B. Meilhan ]{Jean-Baptiste Meilhan} 
\address{Univ. Grenoble Alpes, CNRS, IF, 38000 Grenoble, France}
\email{jean-baptiste.meilhan@univ-grenoble-alpes.fr}
%
%
\begin{abstract} 
It is well known how the linking number and framing can be extracted from the degree $1$ part of the (framed) Kontsevich integral. 
This note gives a general formula expressing any product of powers of these two invariants as combination of coefficients in the Kontsevich integral. 
This allows in particular to express the sum of all coefficients of a given degree in terms of the linking coefficients. The proofs are purely combinatorial. 
\end{abstract} 

\maketitle

\section{Introduction}

The \emph{Kontsevich integral} is a  strong invariant of framed oriented knots and links, which dominates all rational finite type invariants and all Witten--Reshetikhin--Turaev quantum invariants, in the sense that any other factors through it.
It takes values in a certain space of \emph{chord diagrams}, which are copies of the oriented unit circle, endowed with a number of chords, which are pairings of pairwise disjoint points on the circles; chord diagrams naturally come with a degree, which is given by the number of chords. 
Kontsevich defined this invariant in terms of iterated integrals, what can be seen as a far-reaching generalization of the Gauss integral for the linking number of two curves \cite{Kontsevich}.  As a matter of fact, it is well-known that the linking number 
appears as the simplest coefficient in the Kontsevich integral. Specifically, given a framed link $L$, denoting by $C_L[D]$ the coefficient\footnote{The reader who feels nervous about the well-definedness of this notation is referred to \S 3.1.}
 of a chord diagram $D$ in the Kontsevich integral of  $L$, and denoting by $\ell_{ij}$ the linking number of the $i$th and $j$th components, 
we have
\begin{equation}\tag{$1_1$}\label{eq:lk}
 \ell_{ij}(L) = C_L[_i\dessin{0.5cm}{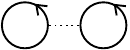}\!\!\textrm{ }_j]. 
\end{equation}
Denoting half the framing of the $i$th component of $L$ by $\ell_{ii}(L)$, it is also well-known that 
\begin{equation}\tag{$2_1$}\label{eq:fr}
  \ell_{ii}(L)=\frac{1}{2} fr_i(L)= C_L[\dessin{0.5cm}{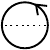}\!\textrm{ }_i].
\end{equation}
Hence the degree $1$ part of the Kontsevich integral of a link is fully characterized by the linking coefficients, \emph{i.e.} the coefficients of the linking matrix. 

The main result of this note is Theorem \ref{thm:main} below, 
which generalizes these two elementary results.  
This provides a general (\emph{i.e.} in all degrees) formula identifying certain combinations of coefficients in the Kontsevich integral in terms of the coefficients of the linking matrix. 
A number of works investigate, in a similar way, how combinations of coefficients in the Kontsevich integral can be expressed in terms of classical invariants of knot theory, see for example \cite{Stanford,HM,Okamoto1997,Okamoto1998,C,CM}, although such results are often only given for low degree terms. 
\medskip

Let $\mathcal{S}_m$ be the set of symmetric matrices of size $m$ with coefficients in $\mathbb{N}$. 
Given $S=(s_{ij})_{i,j}\in \mathcal{S}_m$, we define $\mathcal{D}_S(m)$ as the set of all possible chord diagrams on $m$ circles with exactly $s_{ij}$ chords of type $(i,j)$ for all $i,j$. Here, a \emph{type $(i,j)$ chord} is a chord whose endpoints sit on components $i$ and $j$; in particular, a type $(i,i)$ chord has both endpoints on the $i$th component.

\begin{thm}\label{thm:main}
Let $L$ be an $m$-component framed oriented link in $S^3$ and let  $S=(s_{ij})_{i,j}\in \mathcal{S}_m$. We have 
 $$  \ell_S(L):= \prod_{1\le i\le j\le m} \frac{1}{s_{ij}!} \ell_{ij}(L)^{s_{ij}} = \sum_{D \in \mathcal{D}_S(m)} C_L[D]. $$  
\end{thm}
This general formula has several noteworthy consequences. 

On one hand, if $S$ has a single nonzero entry $n=s_{ij}$ with $i<j$, we obtain a generalization of (\ref{eq:lk}) to all powers of the linking number: 
\begin{equation}\tag{$1_n$}\label{eq:lkn}
 \frac{1}{n!}  \ell_{ij}(L)^n = \sum_{D \in \mathcal{M}^{ij}_n(m)} C_L[D], 
\end{equation}
where $\mathcal{M}^{ij}_n(m)$ denotes the set of all possible degree $n$ chord diagrams on $m$ circles whose $n$ chords are of type $(i,j)$. 
\\
Similarly, if $S$ has a single nonzero entry $n=s_{ii}$ on the diagonal, we obtain that 
\begin{equation}\tag{$2_n$}\label{eq:frn}
 \frac{1}{n!2^n}  fr_i(L)^n = \sum_{D \in \mathcal{I}^i_n(m)} C_L[D], 
\end{equation}
where $ \mathcal{I}^i_n(m)$ denotes the set of all degree $n$ chord diagrams on $m$ circles, such that all $n$ chords are on the $i$th circle. 

On the other hand, the set $\mathcal{D}_{k}(m)$ of all degree $k$ chord diagrams on $m$ circles is partitioned into the sets $\mathcal{D}_S(m)$ for all matrices $S$ in $\mathcal{S}_m$ with $\vert S\vert = k$, where $\vert S\vert = \sum_{1\le i\le j\le m} s_{ij}$ is the \emph{degree} of $S$. Thus we have: 
\begin{equation}\tag{$3$}
 \sum_{D \in \mathcal{D}_{k}(m)} C_L[D] = \sum_{S\in \mathcal{S}_m\textrm{ ; $\vert S\vert = k$}} \ell_S(L). 
\end{equation}
\noindent This expresses the sum of \emph{all} coefficients of degree $k$ in the Kontsevich integral in terms of the linking coefficients.

\begin{acknowledgments}
This work is partially supported by the project AlMaRe (ANR-19-CE40-0001-01) of the ANR. 
The author wishes to thank the referee for pointing out the argument presented in Section \ref{sec:bonus}. He also thanks Georges Abitbol and Benjamin Audoux for inspiring discussions. 
\end{acknowledgments}

\section{The framed Kontsevich integral in a nutshell}\label{sec:K}

We briefly review the combinatorial definition of the framed Kontsevich integral, as given by Le and Murakami in \cite{LM}; 
see also \cite[\S 6]{Ohtsuki}.

A \emph{chord diagram} $D$ on the disjoint union $\dessin{0.25cm}{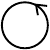}^{\,m}$ of $m$ copies of the oriented circle, is a collection of copies of the unit interval, such that the set of all endpoints is embedded into $\dessin{0.25cm}{D10}^{\,m}$.
We call \emph{chord} any of these copies of the interval.
The  \emph{degree} of $D$ is defined as its number of  chords.
\\
In figures, bold lines depict (portions of) $\dessin{0.25cm}{D10}^{\,m}$, and dashed lines are used for chords.

We denote by $\mathcal{A}(m)$ the $\mathbb{Q}$-vector space generated by all chord diagrams on $\dessin{0.25cm}{D10}^{\,m}$, modulo the \emph{4T relation}: 
 $$ \dessin{1cm}{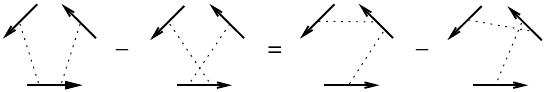}. $$

We now describe the source of the framed Kontsevich integral. 
A framed q-tangle is an oriented tangle, equipped with a framing and a parenthesization on both sets of boundary points. 
Any such tangle can be decomposed into copies of the q-tangles $I$, $X_{\pm}$, $C_{\pm}$ and $\Lambda_{\pm}$ shown in Figure \ref{fig:ehmanitselementarystuff}, along with those obtained by reversing the orientation on any component. 
\begin{figure}[!h]
\includegraphics[scale=0.8]{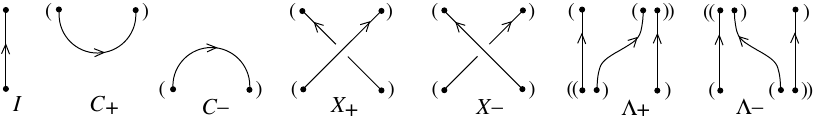}
\caption{The elementary q-tangles $I$, $X_{\pm}$, $C_{\pm}$ and $\Lambda_{\pm}$}\label{fig:ehmanitselementarystuff}
\end{figure}
Such a decomposition is not unique, but a complete set of relations is known, relating any two possible decompositions, see  \cite[Thm.~6.5]{Ohtsuki}. 
The \emph{framed Kontsevich integral} $Z$ can thus be determined by specifying its values on the above q-tangles so that all relations are satisfied.  
This is done as follows. 

We set $Z(I)$ to be the portion of diagram $\uparrow$ without chord. 

For the positive and negative crossings $X_\pm$, we set
 \begin{equation}\label{eq:Xk}
  Z(X_\pm)= \sum_{k\ge 0} \frac{(\pm 1)^k}{2^k k!} X_k\textrm{, where }X_k = \dessin{0.7cm}{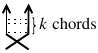}. 
 \end{equation}

Next, set 
$Z(C_\pm)=\sqrt{\nu}$,
where $\nu\in \mathcal{A}(1)$ is the Kontsevich integral of the $0$-framed unknot $U_0$, which was explicitly computed in \cite{BNGT} as follows: 
\begin{equation}\label{eq:nu}
  \nu = \chi\Big(\textrm{exp}_\sqcup\big(\sum_{n\ge 1} b_{2n} W_{2n}\big) \Big), 
\end{equation}
where 
$W_{2n}$ is a \emph{wheel}, that is a unitrivalent diagram of the form $\begin{array}{c}\includegraphics[scale=0.7]{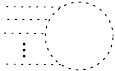}\end{array}$ 
      with $2n$ univalent vertices, 
$b_{2n}$ is the coefficient of $x^{2n}$ in the Taylor expansion of $\frac{1}{2}\textrm{ln}\frac{\textrm{sinh}(x/2)}{x/2}$, 
$\textrm{exp}_\sqcup$ stands for the exponential with respect to the disjoint union of diagrams, 
and the map $\chi$ takes the average over all possible ways of attaching the univalent vertices of a (disjoint union of) wheel(s) to $\dessin{0.25cm}{D10}$, 
      then applies recursively the STU relation below to produce a combination of chord diagrams: 
      $$ \dessin{1cm}{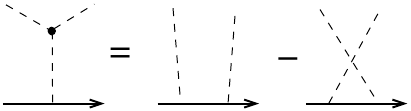}. $$

Lastly, we set $Z(\Lambda_\pm)=\Phi^{\pm 1}$, where $\Phi$ is  a \emph{Drinfeld associator}; 
we do not further discuss here this important ingredient of the construction, as it will play no role in our argument.
The interested reader is referred to \cite[App. D]{Ohtsuki}. 

\section{Linking coefficients and the Kontsevich integral}

The purpose of this section is to prove Theorem \ref{thm:main}. This will be done in Subsection \ref{sec:proof}, after setting some notation and preliminary results in the next two subsections. 

\subsection{Some notation}

Let $A$ be an element of $\mathcal{A}(m)$. Let $D$ be a chord diagram on $m$ circles.
We denote by $C[D](A)$ the coefficient of $D$ in $A$.
In particular, we set 
$$C_L[D] := C[D](Z(L)),$$
for a framed oriented link $L$, that is, $C_L[D]$ denotes \lq the coefficient\rq\, of diagram $D$ in the Kontsevich integral of $L$.
This quantity is of course in general not well-defined since 
an element of $\mathcal{A}(m)$ consists of diagrams subject to the $4T$ relation, but taking an appropriate combination of such coefficients shall yield a link invariant, see  Claim 
\ref{rem:inv} below. 

We denote by $C[D]$ the assignment $A\mapsto C[D](A)$ and, abusing notation, we still denote by  $C[D]$ the precomposition $L\mapsto C_L[D]$ with the Kontsevich integral. 
\medskip 

Let $S=(s_{ij})_{i,j}\in \mathcal{S}_m$ be a symmetric matrix of size $m$ with entries in $\mathbb{N}$.
We set 
$$ \mathcal{L}_S := \sum_{D \in \mathcal{D}_S(m)} C[D], $$
where  $\mathcal{D}_S(m)$ is the set of all chord diagrams on $m$ circles with exactly $s_{ij}$ chords of type $(i,j)$ for all $i\le j$, as defined in the introduction. 
\begin{claim}\label{rem:inv}
This formula yields a well-defined map on $\mathcal{A}(m)$, and in particular defines an $m$-component link invariant. 
\end{claim}

\noindent This is straighforwardly checked using the following general invariance criterion: 
given a collection $\mathcal{D}$ of chord diagrams, the assignement $X := \sum_{D \in \mathcal{D}} C[D]$ defines a map on $\mathcal{A}(m)$ if and only if $X$ vanishes on any linear combination of chord diagrams arising from a $4T$ relation. See for example \cite{C}.
\medskip 

We also recall from the introduction the link invariant associated with the symmetric matrix $S$, 
$$  \ell_S= \prod_{1\le i\le j\le m} \frac{1}{s_{ij} !} \ell_{ij}^{s_{ij}},$$ 
where $\ell_{ij}$ is the linking number between components $i$ and $j$ if $i\neq j$, and $\ell_{ii}:=\frac{1}{2} fr_i$ is half the framing of the $i$th component.

\subsection{Crossing change formula for the invariant $\mathcal{L}_S$}
\label{sec:var}

Let us pick some indices $a,b$ in $\{1,\cdots,m\}$ (possibly with $a=b$). 
Consider two $m$-component links $L_+$ and $L_-$, that are identical away from a small $3$-ball, where they look as follows: 
$$ L_+ = \dessin{0.7cm}{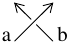}\quad\textrm{and}\quad L_- = \dessin{0.7cm}{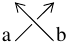}. $$
\noindent 
We stress that this crossing change may involve two strands of either the same ($a=b$) or different ($a\neq b$) components. 

Now let $S=(s_{ij})_{i,j}\in \mathcal{S}_m$. 
By the mere definition of the Kontsevich integral at a crossing, we have
$$\mathcal{L}_S(L_+) - \mathcal{L}_S(L_-)  =  \sum_{j\ge 0} \frac{1}{(2j+1)!2^{2j}} \mathcal{L}_S(D_{2j+1}), $$
where $D_k\in \mathcal{A}(m)$ is obtained from the Kontsevich integral of $L_\pm$ 
by replacing the local contribution of the crossing involved in the crossing change, as given in (\ref{eq:Xk}), by the local diagram $\dessin{0.7cm}{Xk}$ with exactly $k$ parallel chords. 

Set $s:=s_{ab}=s_{ba}$, the entry of the matrix $S$ corresponding to our crossing change. 
In order to slightly simplify our notation, for any $p$ such that $s\ge p\ge 0$ we denote by  $\mathcal{L}_{p}$ the invariant $\mathcal{L}_{S_{p}}$, where $S_{p}$ is the matrix $S$ with the coefficient $s$ replaced by $p$; 
in particular we have $\mathcal{L}_{s}=\mathcal{L}_{S}$. 

Clearly, we have $\mathcal{L}_s(D_k)=0$ for $k>s$, hence the variation formula
\begin{equation}\label{eq:F_n}
\mathcal{L}_S(L_+) - \mathcal{L}_S(L_-)  =  \sum_{j=0}^{\lfloor\frac{s+1}{2}\rfloor} \frac{1}{(2j+1)!2^{2j}} \mathcal{L}_s(D_{2j+1}). 
\end{equation}

When $k\le s$, we have the following. 
\begin{claim}\label{lem:key}
  For all $k$ such that $s\ge k\ge 0$, we have $\mathcal{L}_s(D_k)=\mathcal{L}_{s-k}(D_{0})$.
\end{claim}
\begin{proof}
A degree $n$ chord diagram on $m$ circles that contributes to $\mathcal{L}_s(D_k)$ necessarily 
contains $k$ parallel chords of type $(a,b)$, as imposed by the definition of $D_k$.
The set of all diagrams contributing to $\mathcal{L}_s(D_k)$ is thus obtained by adding $n-k$ chords, with exactly $s-k$ additional chords of type $(a,b)$, in all possible ways. But since these additional chords do not arise from the crossing change, they are attached outside a disk containing the $k$ parallel chords of type $(a,b)$. 
This is thus equivalent to taking the contribution  of \emph{all} chord diagrams in $\mathcal{D}_{S_{s-k}}$.
\end{proof}

By Claim \ref{lem:key}, computing the variation $\mathcal{L}_S(L_+) - \mathcal{L}_S(L_-)$ reduces to computing $\mathcal{L}_{k}(D_0)$ for all $k$. 
This is done in the next lemma.
\begin{lemma}\label{lem:key2}
  We have
  $\mathcal{L}_k(D_0)=\sum_{p=0}^{k} \frac{(-1)^p}{p!2^p} \mathcal{L}_{k-p}(L_+)$
  for all $k$; $s\ge k\ge 0$.
\end{lemma}
\begin{proof}
The proof is by induction on $k$. 
The formula for $k=0$ is clear: 
the invariant $\mathcal{L}_{0}$ vanishes on any diagram with a chord of type $(a,b)$, 
so that $\mathcal{L}_{0}(D_k)=0$ for all $k\ge 1$, 
and by definition of the Kontsevich integral at a positive crossing (\ref{eq:Xk}), we thus have 
 $\mathcal{L}_{0}(L_+)=\mathcal{L}_{0}(D_0)$.
For the inductive step, again by definition of the Kontsevich integral at a positive crossing, we have 
\begin{eqnarray*}
\mathcal{L}_{k}(D_0) & = & \mathcal{L}_k(L_+)-\sum_{j=1}^{k} \frac{1}{j! 2^j} \mathcal{L}_{k}(D_j)\\
                  & = & \mathcal{L}_k(L_+)-\sum_{j=1}^{k} \frac{1}{j! 2^j} \mathcal{L}_{k-j}(D_0)\\
                  & = & \mathcal{L}_k(L_+)-\sum_{j=1}^{k} \frac{1}{j! 2^j} \sum_{i=0}^{k-j} \frac{(-1)^i}{i!2^i} \mathcal{L}_{k-j-i}(L_+),
\end{eqnarray*}
where the second equality follows from Claim \ref{lem:key}, while the third equality uses the induction hypothesis. 
For each $p$ such that $0\le p\le k$, the coefficient of $\mathcal{L}_{k-p}(L_+)$  in the above double sum is then given by 
\begin{eqnarray*}
 -\sum_{j=1}^p \frac{1}{j! 2^j} \times \frac{(-1)^{p-j}}{(p-j)!2^{p-j}}  
  & = & \frac{-1}{p! 2^p} \sum_{j=1}^p (-1)^{p-j} \binom{p}{j} \\
  & = & \frac{-1}{p! 2^p} \Big( \underbrace{\sum_{j=0}^p (-1)^{p-j} \binom{p}{j}}_{=0} -(-1)^p\Big)\\
  & = & \frac{(-1)^p}{p! 2^p}.
\end{eqnarray*}
This concludes the proof.
\end{proof}

\subsection{Proof of Theorem \ref{thm:main}}\label{sec:proof}

We proceed by induction on the degree $\vert S\vert=\sum_{i\le j} s_{ij}$ of $S$.
The base case $n=1$ corresponds to the case where $S$ has a single nonzero entry
$s_{ij}=1$ ($i\le j$), and is given by the well-known formulas (\ref{eq:lk}) and (\ref{eq:fr}) recalled in the introduction. 

Now, assume that the formula holds for all matrices of $\mathcal{S}_m$ of degree $<k$.  
Let $S\in \mathcal{S}_m$ be a degree $k$ matrix. 
Choose some indices $a,b$ such that, in the matrix $S$, the entry $s=s_{ab}$ is nonzero (possibly $a=b$).
Let $L_+$ and $L_-$ be two $m$-component links that differ by a crossing change between components $a$ and $b$, as in Subsection \ref{sec:var}.
Combining (\ref{eq:F_n}) with Claim \ref{lem:key} and Lemma \ref{lem:key2}, we obtain 
$$ \mathcal{L}_S(L_+) - \mathcal{L}_S(L_-)  = \sum_{j= 0}^{\lfloor\frac{s+1}{2}\rfloor} \frac{1}{(2j+1)!2^{2j}} \sum_{k=0}^{s-2j-1} \frac{(-1)^k}{k!2^k} \mathcal{L}_{s-2j-k-1}(L_+).$$
By the induction hypothesis, we have 
$$\mathcal{L}_{s-2j-k-1}(L_+)=\frac{1}{(s-2j-k-1)!} \ell_{ab}^{s-2j-k-1}\prod_{\{i,j\}\neq \{a,b\}} \frac{1}{s_{ij}!} \ell_{ij}^{s_{ij}}.$$
Setting $\ell_0:=\prod_{\{i,j\}\neq \{a,b\}} \frac{1}{s_{ij}!} \ell_{ij}^{s_{ij}}$, we thus have 
$$ \mathcal{L}_S(L_+) - \mathcal{L}_S(L_-) = 
                                           \sum_{j= 0}^{\lfloor\frac{s+1}{2}\rfloor} \frac{\ell_0}{(2j+1)!2^{2j}}  
                                           \sum_{k=0}^{s-2j-1} \frac{(-1)^k}{k!(s-2j-k-1)!2^k}\ell_{ab}(L_+)^{s-2j-k-1}. $$

The coefficient of $\ell_{ab}(L_+)^{s-i}$ in the above formula is given by 
\begin{eqnarray*}
\sum_{j= 0}^{\lfloor\frac{i+1}{2}\rfloor} \frac{\ell_0}{(2j+1)!2^{2j}}.\frac{(-1)^{i-2j-1}}{(i-2j-1)!(s-i)!2^{i-2j-1}} 
     & = & \frac{(-1)^{i+1}\ell_0}{2^{i-1}(s-i)!}\sum_{j= 0}^{\lfloor\frac{i+1}{2}\rfloor} \frac{1}{i!}\binom{i}{2j+1}\\
     & = & \frac{(-1)^{i+1}\ell_0}{2^{i-1}i!(s-i)!}\underbrace{\sum_{j= 0}^{\lfloor\frac{i+1}{2}\rfloor} \binom{i-1}{2j} + \binom{i-1}{2j+1}}_{=2^{i-1}}.
\end{eqnarray*}
This shows that 
$$ \mathcal{L}_S(L_+) - \mathcal{L}_S(L_-) = \ell_0\sum_{i=1}^n \frac{(-1)^{i+1}}{i!(s-i)!} \ell_{ab}^{s-i}. $$

Now, this formula coincides with the variation of the linking invariant $\ell_S$: 
$$ \ell_S(L_+) - \ell_S(L_-) = \ell_0\sum_{i=1}^n \frac{(-1)^{i+1}}{i!(s-i)!} \ell_{ab}^{s-i}. $$
This is easily verified using the binomial formula, noting that $\ell_{ab}(L_-)=\ell_{ab}(L_+)-1$ (in particular, if $a=b$, we indeed have $fr_a(L_-)=fr_a(L_+)-2$).
\medskip

Hence we showed that the invariants $\mathcal{L}_S$ and $\ell_S$ have the same variation formula under a crossing change. 
By a sequence of such operations, any $m$-component link can be deformed into a split union of unknots, each with framing $0$ or $1$ depending on the parity of the framing of the component:   
it remains to check that both invariants take the same value on such links. 

Denote by $U_0$ and $U_1$ the unknot with framing $0$ or $1$, respectively, and let 
$L_0$ be a split union of $m$ unknots, such that the $i$th 
component is a copy of $U_{\varepsilon_i}$, $\varepsilon_i\in\{0,1\}$. 
The framed Kontsevich integral of $L_0$ can be written in $\mathcal{A}(m)$ as a 
disjoint union $Z(L_0)=\sqcup_i Z(U_{\varepsilon_i})$;  
in particular, we may assume that it only contains type $(i,i)$ chords for $1\le i\le m$.

It follows that, if the matrix $S$ contains a nonzero coefficient away from the diagonal, then both invariants  $\mathcal{L}_S$ and $\ell_S$  clearly vanish, and the proof is complete. 

In the case where $S$ is a diagonal matrix, the invariant $\mathcal{L}_S$ of $L_0$ splits as 
$$ \mathcal{L}_S(L_0) = \prod_{i=1}^m \sum_{D \in \mathcal{S}_{s_{ii}}} C_{U_{\varepsilon_i}}[D],$$
where 
$\mathcal{S}_k$ denotes the set of all possible chord diagrams on $\dessin{0.25cm}{D10}$ with $k$ chords. 
The proof then follows readily from the following claim. 
\begin{claim}\label{claimfinal}
For all integer $k$, we have   
\[\textrm{$\sum_{D \in \mathcal{S}_{k}} C_{U_{0}}[D]=0$\,\,\,\,\,\, and\,\,\,\,\,\, 
$\sum_{D \in \mathcal{S}_{k}} C_{U_{1}}[D]=\frac{1}{k!2^k}$.}\]  
\end{claim}
\noindent Indeed, these are the values taken by the invariant $\dfrac{1}{k!}(\frac{1}{2}fr)^k$ on both $U_0$ and $U_1$, thus showing that the invariants $\mathcal{L}_S$ and $\ell_S$ do also coincide on $L_0$ when $S$ is diagonal. 
Hence it only remains to prove Claim \ref{claimfinal} to complete the proof.

\begin{proof}[Proof of Claim \ref{claimfinal}]
For simplicity, set $\mathcal{F}_k(K):= \sum_{D \in \mathcal{S}_k} C_K[D]$ for a knot $K$. 

Let us first compute $\mathcal{F}_k(U_0)$. 
We recalled in (\ref{eq:nu}) the computation of $Z(U_0)=\nu$ of \cite{BNGT}. 
In this computation, given a (disjoint union of) wheel(s) with $k$ univalent vertices attached to $\dessin{0.25cm}{D10}$ in some way, 
applying recursively the STU relation to get a combination of chord diagrams, 
produces an alternate sum with $2^k$ terms, where the coefficients add up to zero. 
This simple observation shows that $\mathcal{F}_k(U_0)=0$.

We now consider $\mathcal{F}_k(U_1)$. 
It is well-known that $\sqrt{\nu}$ commutes with any chord endpoint in a chord diagram (this is a consequence of the 4T relation). 
This can be used to check that 
$$ Z(U_1) = Z(U_0)\sharp \left(\textrm{exp}_\sharp \frac{1}{2}\dessin{0.5cm}{D11}\right) = Z(U_0)\sharp \left( \sum_{k\ge 0} \frac{1}{k!2^k} D_k \right), $$
where $D_k$ denotes the chord diagram on $\dessin{0.25cm}{D10}$ with $k$ parallel chords and where $\sharp$ is the connected sum of chord diagrams. 
From the above computation of $\mathcal{F}_k(U_0)$, we obtain that the only term contributing to $\mathcal{F}_k(U_1)$ is the degree $k$ diagram $D_k$ arising from $\textrm{exp}_\sharp \frac{1}{2}\dessin{0.5cm}{D11}$. Summarizing, we obtain that $\mathcal{F}_k(U_1)=\frac{1}{k!2^k}$.
\end{proof}
 
\subsection{An alternative proof of Theorem \ref{thm:main}}\label{sec:bonus}

We briefly sketch, in this final section, another argument for proving Theorem \ref{thm:main}, that was pointed out by the referee.  

Recall that a \emph{Jacobi diagram} on an oriented $1$-manifold $X$ is a unitrivalent diagram, whose trivalent vertices are equipped with a cyclic ordering of the three incident edges and whose set of univalent vertices is embedded in $X$; each connected component is further assumed to contain at least one univalent vertex. 
In particular, a chord diagram as defined in Section \ref{sec:K} is merely a Jacobi diagram without trivalent vertices. 
The degree of a Jacobi diagram is defined as half its total number of vertices, what agrees with the number of chords for a chord diagram. 
We denote by $\mathcal{J}(X)$ the $\mathbb{Q}$-vector space generated by all Jacobi  diagrams on $X$, modulo the STU relation. We shall consider here the case where $X$ is either the disjoint union $\dessin{0.25cm}{D10}^{\,m}$ of $m$ copies of the oriented circle, or the disjoint union $\uparrow^{\,m}$ of $m$ copies of the oriented interval.  
Both $\mathcal{J}(\dessin{0.25cm}{D10}^{\,m})$ and $\mathcal{J}(\uparrow^{\,m})$ have a coalgebra structure, where the coproduct $\Delta(J)$ of a diagram $J$ is given by the sum of all ways of splitting $J$ in a disjoint union of two Jacobi diagrams; this actually endows $\mathcal{J}(\uparrow^{\,m})$ with a Hopf algebra structure, with product given by stacking intervals. Picking an oriented interval in each component of $\dessin{0.25cm}{D10}^{\,m}$ yields a canonical map $\iota: \mathcal{J}(\uparrow^{\,m})\rightarrow \mathcal{J}(\dessin{0.25cm}{D10}^{\,m})$. 
There is a natural isomorphism $\varphi : \mathcal{J}(\dessin{0.25cm}{D10}^{\,m})\rightarrow \mathcal{A}(m)$, given by expressing a Jacobi diagram as a linear combination of chord diagrams using the STU relation, see \cite[Thm.~6]{BNvass}. 

We also need the space $\mathcal{B}(m)$ of labeled Jacobi diagrams, which is the $\mathbb{Q}$-vector space spanned by unitrivalent diagrams with univalent vertices labeled by $\{1,\cdots,m\}$, modulo the AS and IHX relations, see  \cite{BNvass}. This is a graded Hopf algebra, with product given by the disjoint union $\sqcup$, and graded by half the number of vertices. 
As a diagrammatic analogue of the PBW isomorphim, we have a graded Hopf algebra isomorphism $\chi: \mathcal{B}(m)\rightarrow \mathcal{J}(\uparrow^{\,m})$, which acts by averaging all ways of attaching the $k$-labeled univalent vertices of a diagram along the $k$th oriented interval of $\uparrow^{\,m}$. 

Given a framed link $L$ in $S^3$, a fundamental property of the Kontsevich integral is that $Z(L)$ is a \emph{group-like} element in $\mathcal{J}(\dessin{0.25cm}{D10}^{\,m})$, see \cite[Thm.~3.7]{LMO}. 
More precisely, we have that 
$$ Z(L) = \iota\circ \chi \left( \textrm{exp}_\sqcup\Bigg(\sum_{1\le i\le j\le m} \ell_{ij} I_{ij} \Bigg)\sqcup T\right)\in \mathcal{J}(\dessin{0.25cm}{D10}^{\,m}), $$
where the $\ell_{ij}=\ell_{ij}(L)$ are the linking coefficients as above, $I_{ij}$ denotes the (dashed) interval with endpoints labeled by $i$ and $j$, and where $T$ is a linear combination of possibly disconnected labeled Jacobi diagrams, each having at least one trivalent vertex.  
%
Now, the formula of Theorem \ref{thm:main} can be derived as follows (here we freely use of the notation of the Section \ref{sec:K}). 
Given $S=(s_{ij})_{i,j}\in \mathcal{S}_m$, let $\mathcal{C}_S:\mathcal{J}(\dessin{0.25cm}{D10}^{\,m})\rightarrow \mathbb{Q}$ be the map defined by $\mathcal{C}_S(J)=\sum_{D \in \mathcal{D}_S(m)} C_{\varphi(J)}[D]$; for a link $L$ in $S^3$, note that $\mathcal{C}_S(Z(L))$ is precisely the sum of coefficients $\sum_{D \in \mathcal{D}_S(m)} C_{L}[D]$ involved in Theorem \ref{thm:main}. 
As in the proof of Claim \ref{claimfinal}, we have that  $\mathcal{C}_S(T)=0$ by the STU relation. Hence the above expression for $Z(L)$ gives 
$$ \mathcal{C}_S(Z(L)) = \mathcal{C}_S\circ \iota\circ \chi \left(\textrm{exp}_\sqcup\Bigg(\sum_{i\le j} \ell_{ij} I_{ij} \Bigg)\right) = \sum_n \frac{1}{n!} \mathcal{C}_S\circ \iota\circ \chi \Bigg(\sum_{i\le j} \ell_{ij} I_{ij} \Bigg)^{\sqcup n}. $$
All terms in the right-hand sum are zero, except when $n=\vert S\vert=\sum_{i\le j}s_{ij}$ is the degree of $S$ (where we have the \lq right  number\rq~ of chords). For $n=\vert S\vert$, we have that 
 $$ \frac{1}{\vert S\vert!} \Bigg(\sum_{i\le j} \ell_{ij} I_{ij} \Bigg)^{\sqcup \vert S\vert} 
   = \prod_{i\le j} \frac{1}{s_{ij}!} \ell_{ij}^{s_{ij}} I_{ij}^{\,\sqcup s_{ij}} = \ell_S(L) I_{ij}^{\,\sqcup s_{ij}}. 
   $$ 
Since $\mathcal{C}_S\circ \iota\circ \chi\left(I_{ij}^{\,\sqcup s_{ij}}\right)=1$ as a direct consequence of the definitions, the conclusion follows.
\bibliographystyle{abbrv}
\bibliography{frlk}

\begin{thebibliography}{10}

\bibitem{BNvass}
D.~Bar-Natan.
\newblock On the {V}assiliev knot invariants.
\newblock {\em Topology}, 34(2):423--472, 1995.

\bibitem{BNGT}
D.~{Bar-Natan}, S.~{Garoufalidis}, L.~{Rozansky}, and D.~P. {Thurston}.
\newblock {Wheels, wheeling, and the Kontsevich integral of the unknot.}
\newblock {\em {Isr. J. Math.}}, 119:217--237, 2000.

\bibitem{C}
A.~Casejuane.
\newblock {\em Formules combinatoires pour les invariants de n\oe uds et de
  vari\'et\'es de dimension 3}.
\newblock PhD thesis, Universit\'e Grenoble Alpes, 2021.

\bibitem{CM}
A.~Casejuane and J.-B. Meilhan.
\newblock Universal invariants, the conway polynomial and the
  casson-walker-lescop invariant (2020), arXiv:2003.05527.

\bibitem{HM}
N.~Habegger and G.~Masbaum.
\newblock The {K}ontsevich integral and {M}ilnor's invariants.
\newblock {\em Topology}, 39(6):1253--1289, 2000.

\bibitem{Kontsevich}
M.~{Kontsevich}.
\newblock {Vassiliev's knot invariants.}
\newblock In {\em {I. M. Gelfand seminar. Part 2: Papers of the Gelfand seminar
  in functional analysis held at Moscow University, Russia, September 1993}},
  pages 137--150. Providence, RI: American Mathematical Society, 1993.

\bibitem{LM}
T.~Q.~T. {Le} and J.~{Murakami}.
\newblock {The universal Vassiliev-Kontsevich invariant for framed oriented
  links.}
\newblock {\em {Compos. Math.}}, 102(1):41--64, 1996.

\bibitem{LMO}
T.~T.~Q. Le, J.~Murakami, and T.~Ohtsuki.
\newblock On a universal perturbative invariant of {$3$}-manifolds.
\newblock {\em Topology}, 37(3):539--574, 1998.

\bibitem{Ohtsuki}
T.~{Ohtsuki}.
\newblock {\em {Quantum invariants. A study of knots, 3-manifolds, and their
  sets.}}
\newblock Singapore: World Scientific, 2002.

\bibitem{Okamoto1997}
M.~{Okamoto}.
\newblock {Vassiliev invariants of type 4 for algebraically split links.}
\newblock {\em {Kobe J. Math.}}, 14(2):145--196, 1997.

\bibitem{Okamoto1998}
M.~{Okamoto}.
\newblock {On Vassiliev invariants for algebraically split links.}
\newblock {\em {J. Knot Theory Ramifications}}, 7(6):807--835, 1998.

\bibitem{Stanford}
T.~{Stanford}.
\newblock {Some computational results on mod 2 finite-type invariants of knots
  and string links}.
\newblock In {\em {Invariants of knots and 3-manifolds. Proceedings of the
  workshop, Kyoto, Japan, September 17--21, 2001}}, pages 363--376. Coventry:
  Geometry and Topology Publications, 2002.

\end{thebibliography}

\end{document}